\documentclass[12pt]{amsart}
\usepackage{amscd,amssymb,amsfonts}
\usepackage[arrow,matrix,graph,frame,poly,arc]{xy}

\newtheorem{theorem}[subsection]{Theorem}
\newtheorem{lemma}[subsection]{Lemma}
\newtheorem{prop}[subsection]{Proposition}
\newtheorem{corollary}[subsection]{Corollary}

\theoremstyle{definition}
\newtheorem{remark}[subsection]{Remark}
\newtheorem{definition}[subsection]{Definition}
\newtheorem{example}[subsection]{Example}
\newtheorem*{ack}{Acknowledgment}

\numberwithin{equation}{section}
\newcommand{\A}{{\mathcal A}}
\newcommand{\B}{{\mathcal B}}
\newcommand{\RR}{{\mathcal R}}

\newcommand{\Z}{\mathbb{Z}}

\renewcommand{\k}{\Bbbk}

\DeclareMathOperator{\Ker}{Ker}

\DeclareMathOperator{\rk}{rk}
\DeclareMathOperator{\gr}{gr}
\DeclareMathOperator{\im}{im}
\DeclareMathOperator{\coker}{coker}

\DeclareMathOperator{\id}{id}

\begin{document}
%\date{Jan 30, 2008}

\title[Cohomology rings and formality properties of nilpotent groups]{%
Cohomology rings and formality properties of nilpotent groups}

\author[A.~D.~Macinic]{Anca Daniela M\u acinic$^*$}
\address{Inst. of Math. Simion Stoilow,
P.O. Box 1-764,
RO-014700 Bucharest, Romania}
\email{Anca.Macinic@imar.ro}

\thanks{$^*$Partially supported by the CEEX Programme of
the Romanian Ministry of Education and Research, contract
2-CEx 06-11-20/2006.}



\begin{abstract}

We introduce partial formality and relate resonance with
partial
formality properties.
For instance, we show that for finitely generated nilpotent groups
that are
k-formal, the resonance varieties are trivial up to degree k.
We also show that the cohomology ring of a nilpotent k-formal group is
generated in degree 1, up to degree k+1; this criterion is necessary and
sufficient for 2-step nilpotent groups to be k-formal.
We compute resonance varieties for Heisenberg-type groups and deduce the degree 
of partial formality for this class of groups. 
 
\end{abstract}

\maketitle

\tableofcontents

\section{Introduction and statement of results}
\label{sect:intro}

A space which has the minimal model isomorphic to the minimal model of its cohomology ring 
is called formal. In other words, the "rational homotopy type" of the space is a formal consequence 
of its cohomology ring. Compact K\"ahler manifolds (in particular, smooth complex projective varieties) 
are important examples of formal spaces. See Deligne-Griffiths-Morgan-Sullivan \cite{DGMS} for details.

It seems natural to introduce a more relaxed version, namely $k$-\textit{formality}
(see Definition \ref{s-formality}), following \cite{DGMS}. For $k=1$, it coincides with 
the usual notion of $1$-formality present in the literature (see \cite{ABC}). 
See \cite{ABC} for the equivalence between $1$-formality of a space and quadratic presentability 
of the Malcev Lie algebra associated to the fundamental group of that space, and also \cite{CT}.

Our notion of partial formality is strictly weaker, despite the terminology, than the one introduced 
by Fern\'{a}ndez-Mu\~{n}oz in \cite[ Definition 2.2]{FM}. For instance, in the case of nilmanifolds, 
formality is equivalent to $1$-formality in the sense of \cite{FM} (see \cite[Lemma 2.6]{FM}), 
which is not the case with our definition (see Corollary \ref{c1}). This is due to the fact that 
the test of partial formality in \cite{FM} is global, in the sense that it involves the whole minimal model,
whereas our $k$-formality test is a finite one, using only information provided by the $k$-minimal model, 
up to degree $k+1$; see Proposition \ref{art}(1), and the discussion following it.

It is well-known that $1$-formality (in our sense) is the first general obstruction in 
the Serre problem regarding the characterization of \textit{projective groups} 
(fundamental groups of smooth projective complex varieties).
A difficult particular case of this problem turns out to be the one of nilpotent groups. 
A positive answer is given by Campana \cite{C} for a certain class of $2$-step nilpotent groups, 
the Heisenberg groups $\mathcal{H}_{n \geq 4}$ (see Definition \ref{heisenberg}).
As for the remaining Heisenberg groups, $\mathcal{H}_{1}$ does not pass the $1$-formality criterion 
(see Corollary \ref{c1} for a slightly more general result);  $\mathcal{H}_{2}$ and $\mathcal{H}_{3}$ 
are also non-projective groups  (see Carlson-Toledo \cite[Corollary 4.5]{CT}). 
When passing to $3$-step nilpotent groups, 
the answer is not known, according to \cite{CT}.

 We approach in Corollary \ref{T} the solution to the Serre problem given by Campana, 
 to point out some  homotopic features of the projective smooth 
 complex varieties constructed in \cite{C}.
 This is actually a consequence of a more general result (where partial formality for 
 a group $G$ is defined via the classifying space $K(G,1)$).
 
\begin{theorem}
\label{Thm}
Let $M$ be a $k$-formal space such that $\pi_1(M)$ is not $k$-formal $(k\ge 2)$. Then 
there exists $2 \leq i \leq k$ such that $\pi_i(M) \neq 0$.
\end{theorem}

Note that the above result no longer holds if partial formality is taken in the sense of \cite{FM}; 
see Remark \ref{remFM}.

Some of the results presented in Section \ref{obstr}, such as passing from partial to 
full formality (Proposition \ref{k+2}) and Proposition \ref{art}\eqref{ii} are inspired by 
the similar results obtained in the $1$-connected case by Papadima in \cite{P2}.

For nilpotent groups we find obstructions to (partial) formality involving either 
generators of the cohomology ring (up to a degree) or certain \textit{resonance varieties} 
(see Definition \ref{var-res}) associated to the cohomology ring.

\begin{theorem}
\label{B}
Let $G$ be a finitely generated nilpotent group.
\begin{enumerate}
\item \label{motz} 
If $G$ is k-formal , then $H^{\leq k+1}(G)$ is generated as an algebra by $H^1(G)$.
\item \label{converse}
For 2-step nilpotent groups, the converse also holds.
\end{enumerate}
\end{theorem}

For $k=1$, the first part of Theorem \ref{B} follows from \cite{ABC}, Lemma 3.17 on page 35.
The second part of Theorem \ref{B}, for the case $k=1$, follows from \cite{CT}, Corollary $0.2.1$.
As explained in Remark \ref{rem=noteq}, the second part may fail, if $G$ is not
2-step nilpotent, even for $k=1$. 

Another obstruction to partial formality can be expressed in terms of the resonance varieties
$\RR_1^i (G)\subseteq H^1(G, \k)$, over a field $\k$ of characteristic zero.

\begin{theorem}
\label{zerores}
Let $G$ be a nilpotent, finitely generated, $s$-formal group. 
Then the resonance varieties of $G$ are trivial up to degree $s$, that is,
$\RR_1^i (G)\subseteq \{ 0\}$ for $i\le s$.
\end{theorem}

Note that this result, via our Lemma \ref{non}, generalizes  \cite[Lemma 2.4]{CT}, 
corresponding to the case $s=1$. 
 
For fundamental groups of complements of arrangements of complex hyperplanes, 
it is well-known that 1-formality holds. It turns out 
that our nilpotency test from Theorem \ref{zerores}, via resonance, is faithful (see Example \ref{arr}).

In Section \ref{section:heisenberg} we make an analysis of the formality properties for 
Heisenberg-type groups from a double perspective - generators of the cohomology ring and resonance varieties.

\section{Partial minimal models and formality properties}
\label{sect:2step}

For D. Sullivan's theory of minimal models  we refer to \cite{S} (see also 
\cite{DGMS}, \cite{FHT}, \cite{HS} and \cite{M}).

Let $({A}^{*}, d_A)$ be a differential graded algebra (D.G.A.) over a field $\k$ of characteristic 
zero, such that $H^0(A^*, d_A)$ is the ground field. A minimal model for ${A}^{*}$ is a minimal 
D.G.A. $(\mathcal{M}, d_{\mathcal{M}})$ such that there exists a morphism of D.G.A.'s 
$\rho : \mathcal{M} \longrightarrow A^{*}$ inducing isomorphism on cohomology. There is a 
unique (up to isomorphism) $\mathcal{M}=\mathcal{M}(A)$ satisfying the conditions in the definition.

Let $K$ be a space having the homotopy type of a connected simplicial complex. We call 
the minimal model of $K$, denoted $\mathcal{M}(K)$, the minimal model associated to the D.G.A. 
of p.l. forms  $\Omega ^{*}(K)$.

A D.G.A. $A^*$ as above is called formal if there exists a D.G.A. morphism  
$(\mathcal{M}(A), d_{\mathcal{M}}) \longrightarrow (H^{*}(A),d=0)$ 
which induces isomorphism in cohomology.

$K$ is formal if and only if the minimal model of $K$ is a formal D.G.A, i.e. 
$ \mathcal{M}(K)=\mathcal{M}(H^{*}(K), d=0)$.

\begin{definition}
\label{k-mod}
 A minimal algebra $\mathcal{M}$ generated by elements of degree $\leq k$ is called 
 a $k$-minimal model of a D.G.A. $(A^*, d_A)$ if there exists a D.G.A. map 
 $\rho: \mathcal{M} \longrightarrow A$ such that it  induces in cohomology isomorphisms 
 up to degree $k$ and a monomorphism in degree $k+1$. Again, such an object exists and is 
 uniquely determined, up to isomorphism, for any D.G.A. $(A^*, d_A)$. 
 Notation: $\mathcal{M}=\mathcal{M}_k(A)$.
\end{definition}
 
 \begin{remark}
 \label{scale}
 If $\mathcal{M}=(\wedge V,d)$ is a minimal algebra, then 
 $\mathcal{M}_k(\mathcal{M})= (\wedge V^{\leq k},d)$. For a space $K$, 
 set $\mathcal{M}_k(K):=\mathcal{M}_k(\mathcal{M}(K))$.
 \end{remark}
 
 \begin{example}
 \label{ex:gr}
 Assume $H^*(K)=\wedge (x_1, \dots, x_n)$, as rings.
 Then the minimal model of $K$ is $\mathcal{M}(K)=(\wedge (x_1, \dots, x_n), d=0)$, 
 hence $K$ is formal. 
 If deg$(x_i)=1$, for all $i$, then $\mathcal{M}_1(K)=\mathcal{M}(K)$; it follows from \cite{S} 
 that the rational associated graded Lie algebra of $G=\pi_1(K), \; \gr(G) \otimes \mathbb{Q}$, is abelian,
 concentrated in degree $1$ (i.e. $\gr(G)_{\ge 2} \otimes \mathbb{Q}=0$).
 \end{example}
 
 \begin{definition}
 \label{DGA k-formal}
 A D.G.A. $(A^*,d_A)$ is called $k$-{\em formal} if there exists a D.G.A. morphism  
 $(\mathcal{M}_k(A), d) \to (H^{*}(A),0)$ which induces isomorphisms in cohomology up to degree $k$ 
 and a monomorphism in degree $k+1$.
 \end{definition}
 
\begin{definition}
\label{s-formality}
The space $K$ is called $k$-{\em formal} if $\mathcal{M}(K)$ is a $k$-formal D.G.A. 
In other words, $\mathcal{M}_k(K)=\mathcal{M}_k(H^*(K),0)$.
\end{definition}

A group $G$ is called formal (respectively $k$-formal) if the associated Eilenberg-MacLane 
space $K(G,1)$ is formal (respectively $k$-formal).
This convention (replace the  $K(G,1)$ space by  the group $G$) will be used from now on.

Note that a formal space is $k$-formal, for any $k$. 
A partial converse will be proved later, in Proposition \ref{k+2}.

Recall from \cite{W} that a continuous map between connected $CW$-complexes 
$f:X \longrightarrow Y$ is called a 
$k$-homotopy equivalence if it induces isomorphisms on homotopy groups up to degree $k-1$ 
and a surjection in degree $k$.
 Up to homotopy, we can see $f$ as an inclusion. Then from the long exact homotopy sequence 
 associated to the pair $(Y,X)$ we get $\pi _{\leq k}(Y,X)=0$, hence, by  the Hurewicz theorem 
 (relative version), $H_{\leq k}(Y,X)=0$. It follows that $H^{\leq k}(Y,X)=0$, and we can apply 
 the long exact cohomology sequence of the pair to conclude that we have 
 $H^{\leq k-1}(X)\cong H^{\leq k-1}(Y)$  and an injection $H^k(Y)\hookrightarrow H^k(X)$, 
 both induced by the inclusion $f$.
 Therefore a $k$-homotopy equivalence is a \textit{homology $k$-equivalence} 
 (that is, a map which induces isomorphisms on cohomology groups 
 up to degree $k-1$ and a monomorphism in degree $k$).

\begin{prop}
\label{L}
Let $f:X \longrightarrow Y$ be a homology $k$-equivalence. Then 
\begin{equation}
\label{k-1}
\mathcal{M}_{k-1}(X) \cong \mathcal{M}_{k-1}(Y)
\end{equation}
and 
\begin{equation}
\label{m^k-1}
\mathcal{M}_{k-1}(H^*(X),0) \cong \mathcal{M}_{k-1}(H^*(Y),0)
\end{equation}
\end{prop}

\begin{proof}
 Let  $\rho :\mathcal{M}_{k-1}(Y) \longrightarrow \Omega^{*}(Y) $ be a map as 
 in the definition of the $k$-minimal model. Since $f$ is a homology $k$-equivalence, the map  
 $f^* \circ \rho: \mathcal{M}_{k-1}(Y) \longrightarrow \Omega^*(X)$ satisfies 
 the conditions from the definition of the $k$-minimal model, so $\mathcal{M}_{k-1}(Y)$ 
 is also the $(k-1)$-minimal model of $X$.
 A similar proof shows that $\mathcal{M}_{k-1}(H^*(X),0) \cong \mathcal{M}_{k-1}(H^*(Y),0)$. 
\end{proof}

\begin{corollary}
\label{form}
Assume $f$ is a homology $k$-equivalence. Then $X$ is $(k-1)$-formal 
if and only if $Y$ is $(k-1)$-formal.
\end{corollary}

\begin{remark}
\label{remFM}
The previous result is similar to \cite[Theorem 5.2(i)]{FM}, where only 
"$Y$ $(k-1)$-formal $\Rightarrow$ $X$ $(k-1)$-formal" is shown, but uses a different notion 
of partial formality, as explained in Remark \ref{FernandezMunoz}. With the definition 
of partial formality in the sense of \cite{FM}, Theorem \ref{Thm} no longer holds. 
This can be seen by considering $M$ to be the projective smooth complex variety with 
fundamental group a Heisenberg group $\mathcal{H}_{n \geq 4}$ (see Definition \ref{heisenberg}), 
constructed by Campana (\cite{C}). Then $M$ is formal, hence $k$-formal in the sense of 
\cite[Definition 2.2]{FM}, for any $k$, but $\pi_1 (M)$ is not even $1$-formal in the sense of \cite{FM} 
(by \cite[Lemma 2.6]{FM}). Were Theorem \ref{Thm} true for $k=2$ it would imply that $\pi_2(M) \neq 0$. 
One can deduce that $\pi_2(M)=0$ from the construction of $M$ 
(see \cite{C}, or \cite[Section 5]{CT}), when $n \geq 6$.
\end{remark}

Theorem \ref{Thm} follows from the result below:
\begin{theorem}
\label{thm:1.1strong}
Assume either $M$ is a $k$-formal space such that $\pi_1(M)$ is not $k$-formal or 
$M$ is not $k$-formal and $\pi_1(M)$ is $k$-formal, where $k\ge 2$. Then there exists 
$2 \leq i \leq k$ such that $\pi_i(M) \neq 0$.
\end{theorem}

\begin{proof}
Assume $\pi_i(M)=0, \; \forall \; 2 \leq i \leq k$ and set $\pi_1(M):=H$.
Consider the classifying map $f:M \longrightarrow K(H,1)$ such that $\pi_1(f)=\id_{H}$. 
This map is then a $(k+1)$-homotopy equivalence. Then $M$ is $k$-formal if and only if $H$ is $k$-formal, 
by Corollary \ref{form}, contradicting our hypothesis.
\end{proof}

\section{Obstructions to partial formality}
\label{obstr}

 We begin by giving an alternative characterization of (partial) formality, along the lines 
 from \cite{DGMS} and \cite{FM}.
 
 We will need several basic properties of the bigraded minimal model of a connected, 
 graded-commutative algebra $H^*$, extracted from \cite{HS}.

As an algebra, $\mathcal{B}= \wedge Z$, where $Z$ is bigraded by 
$Z= \oplus_{i \geq 0,\;p \geq 1}Z^p_i$. The differential $d$ has degree $+1$ with respect to 
upper degrees and is of degree $-1$ with respect to lower degrees. Moreover,

\begin{equation}
\label{basic}
H_{+}(\mathcal{B}, d)=0
\end{equation}

 The $k$-minimal model of $\mathcal{B}$ will be denoted by 
 $_{k} \mathcal{B}:= (\wedge Z^{\leq k},d)$.
 
 For a D.G.A. $(\wedge V, d)$, set $C^*= \Ker (d|_{V^*})$.

\begin{prop}
\label{art}
The following hold.
\begin{enumerate}
\item \label{i} 
A D.G.A. $(A^*, d_A)$  is $k$-formal ($1 \leq k \leq \infty$) 
if and only if it has a $k$-minimal model $\mathcal{M}_k=( \wedge V^{\leq k}, d)$ with 
a decomposition $V^{\leq k}=C^{\leq k} \oplus N^{\leq k}$ such that 
$(N  \cdot \mathcal{M}_k \cap \Ker\;d)^{\leq k+1} \subset d\mathcal{M}_k$, 
where $N=\oplus_{1 \leq i \leq k}N^i$.
\item \label{ii}
Assume  $\mathcal{M}_k=( \wedge V^{\leq k}, d)$ has the property from Part \eqref{i} above.
Let $\phi: \wedge C^{\leq k} \longrightarrow H^*(\mathcal{M}_k)$ be the map of 
graded algebras which associates to an element in $\wedge C^{\leq k}$  its cohomology class. 
Then $\phi$ is surjective up to degree $k+1$.
\item \label{iii}
If $M$ is a $k$-formal space and a rational $K(\pi,1)$, then the cohomology algebra of $M$ 
is generated by $H^1(M)$, up to degree $k+1,\; i.e.$ $H^{\leq k+1}(M)=(H^1(M))^{ \leq k+1}$.
\end{enumerate}
\end{prop}
 
\begin{proof} 
\eqref{i} 
 If $A^*$ is $k$-formal, then there is a D.G.A. isomorphism 
 $ \mathcal{M}_k(A) \cong \;_{k}\mathcal{B}$, where $\mathcal{B}$ is the bigraded model 
 of $H^*(A)$. The required decomposition of $Z^{\le k}$ is as follows: $C^q=Z_0^q$ and 
 $N^q= \oplus_{i>0}Z^q_i$, for $q \leq k$. Let $ x \in N \cdot_k\mathcal{B} \cap \Ker\;d$, 
 be homogeneous of upper degree $q \leq k+1$. The lower degree of each component of $x$ 
 is strictly positive, hence $x=d(z)$, for some $z \in \mathcal{B}^{q-1}= \;_k\mathcal{B}^{q-1}$, 
 since  $H_+(\mathcal{B})=0$ and $q \leq k+1$.
 
 To prove the converse claim, define a G.A. map 
 \begin{equation}
 \label{ro}
 \rho: \;\mathcal{M}_k  \longrightarrow H^*(\mathcal{M}_k)
 \end{equation}
 by $c \mapsto [c]$  and $n \mapsto 0$, for $c \in C^{\leq k}, \; n \in N^{\leq k}$.

We begin by showing that $\rho$ is in fact a D.G.A. map, that is $\rho(d(v))=0$, 
for any $v \in V^{\leq k}$.

For  $c \in C^{\leq k}$ this is true, since $d(c)=0$.

Take $n\in N^{\leq k}$. We can write 
$d(n)= \overline{n} +\overline{c}, \; \overline{n} \in N \cdot \mathcal{M}_k, \; \overline{c} \in \wedge C^{\leq k}$. 
Then $\overline{n} \in (N \cdot \mathcal{M}_k \cap \Ker\;d)^{\leq k+1}$, so $\overline{n}$ 
is a boundary in $\mathcal{M}_k$, which implies $\rho(d(n))=[\overline{c}]=0$.

Let us prove now the injectivity of the map $H^q(\rho)$, for $q \leq k+1$. Take 
$\alpha \in \; \mathcal{M}_k= \wedge V^{\leq k}$ homogeneous of degree $q$, such that 
$H^q(\rho)([\alpha])=0$ and write $\alpha$ as a sum $\alpha= \alpha_1+ \alpha_2$, where 
$\alpha_1 \in \wedge C^{\leq k}$ and $\alpha_2 \in N \cdot \mathcal{M}_k$. It follows that 
$d(\alpha_2)=d(\alpha-\alpha_1)=0$, hence $\alpha_2=d(z), \; z \in \mathcal{M}_k$, so 
$[\alpha]=[\alpha_1]$. This last equality, together with $H^q(\rho)([\alpha])=0$ implies 
$[\alpha]=[\alpha_1]=0$, hence $H^q(\rho)$ is injective.

We prove the surjectivity of $H^q(\rho)$, for $q \leq k+1$. Let $[\alpha] \in H^q(\mathcal{M}_k)$, 
with $\alpha$ written just as before as a sum $\alpha=\alpha_1+\alpha_2$. Again $\alpha_2$ is exact 
in $\mathcal{M}_k$, and $H^q(\rho)([\alpha_1])= [\alpha]$, hence surjectivity is proven.

To end the proof, consider the composition of $\rho$ with the map induced in cohomology by 
the map in Definition \ref{k-mod}.

\eqref{ii}
The proof of the fact that $\phi$ is surjective up to degree $k+1$ goes  the same way as 
the proof of the surjectivity of $H^q(\rho)$ in Part \eqref{i}.

\eqref{iii} We infer from $k$-formality the existence of a $k$-minimal model of $M,\; \mathcal{M}_k$, 
having the property from Part \eqref{i}. Since $M$ is a rational $K(\pi,1)$, 
$V^{\leq k}=V^1$ and $C^{\le k}=C^1$. Moreover, $\mathcal{M}_k$ is actually the minimal model of $M$. 
Our claim follows by considering the map $\phi: \wedge C^1 \longrightarrow H^*(M)$ 
defined in Part \eqref{ii}.
\end{proof}

Proposition \ref{art}(1) 
shows that our notion of $k$-formality from Definition \ref{DGA k-formal} is less
restrictive than the one from \cite[Definition 2.2]{FM}. Actually, the requirements in
\cite{FM} are strictly stronger than ours;
see Remark \ref{FernandezMunoz}. In \ref{art}(2) we assume less than in \cite{FM}, and we obtain more.
  
\begin{definition}
\label{2-step}
A 2-step nilpotent group is a group $G$ such that 
$[G, [G,G]]=0$, where $[,]$ denotes the group commutator.
\end{definition}

\begin{remark}
\label{alt}
To say that $G$ is a finitely generated 2-step nilpotent group is equivalent to say that 
$G$ is a central  extension of some finite rank abelian group by another finite rank abelian group.
\end{remark}

{\bf Proof of Theorem \ref{B}}
\eqref{motz}
Follows from Proposition \ref{art} Part \eqref{iii}, due to the fact that 
$K(G,1)$ is a rational $K(\pi,1)$.

\eqref{converse}
  Since $G$ is two-step nilpotent, we can choose the generators of the minimal model 
  such that $\mathcal{M}(G)= \wedge (x_1, \dots, x_m) \otimes \wedge (y_1, \dots, y_n)$, 
  deg$(x_i)=$ deg$(y_j)=1$, 
  $d(y_j) \in \wedge^2(x_1, \dots ,x_m), \forall j$ and $C^1=<x_1, \dots, x_m>$.
 
 Define then a map $\phi$ (see Definition \ref{DGA k-formal}),  
 $\phi : (\mathcal{M},d) \longrightarrow (H^*(G),0)$
 by $x_i \mapsto [x_i]$ and $y_j \mapsto 0$. It is easy to see that $\phi$ is a D.G.A. 
 morphism and $H^1(\phi)$ is the identity. As a consequence of the fact that 
 $H^{\leq k+1}(G)=(H^1(G))^{ \leq k+1}$, $H^{\leq k+1}(\phi)$ is also the identity, 
 being induced by $H^1(\phi)$. This proves the $k$-formality of $G$. \hfill $\square$

\begin{corollary}[\cite{Has}]
\label{nilmanifold} 
If $G$ is a finitely generated nilpotent group, then $G$ is formal 
if and only if it is rationally abelian.
\end{corollary}

\begin{proof}
To justify the less obvious implication, assume the minimal model of the 
finitely generated nilpotent group $G$ is 
$\mathcal{M}= \wedge (x_1, \dots, x_n, y_{n+1}, \dots, y_p)$, with $C^1=<x_1, \dots, x_n>$. 
One knows that $\dim H^p(\mathcal{M})=1$, since $\mathcal{M}$ has the same cohomology 
as a $p$-dimensional nilmanifold. On the other hand, from Theorem \ref{B}, Part \eqref{motz}, 
for $k= \infty$, we get that $H^*(\mathcal{M})$ is a quotient of the algebra $\wedge^*(x_1, \dots, x_n)$, 
hence $p=n$, which proves our claim.  
\end{proof}

We will also need the following lemma:
 
\begin{lemma}
\label{sem}
Let $(A^*,d_A)$ be a D.G. algebra. Any $k$-minimal model 
$\mathcal{M}_k=( \wedge V^{\leq k},d) \overset{\phi_k}{\longrightarrow} (A^*,d_A)$ 
can be extended to a $(k+1)$-minimal model 
$\mathcal{M}_{k+1}=( \wedge V^{\leq k+1},d) \overset{ \phi_{k+1}}\longrightarrow (A^*,d_A)$ 
such that for any $v \in V^{k+1}, \; d(v) \in d(\mathcal{M}_k)$ if and only if $d(v)=0$. 
Moreover, $V^{k+1}= \oplus_{i \geq 0}V^{k+1}_i$ and $d(v)=0$ if and only if $v \in V_0^{k+1}$.
\end{lemma}

\begin{proof}
 Following a standard technique we gradually define the vector space of 
 degree $k+1$ generators $V^{k+1}$.
 
 Set $V_0^{k+1}:= \coker H^{k+1}(\phi_k)$, with $d|_{V_0^{k+1}}=0$. Assume we have already  
 defined $V_i^{k+1}$ such that there is an extension 
 $\phi_{ki}: \wedge (V^{\leq k} \oplus V^{k+1}_{\leq i}) \rightarrow (A^*, d_A)$ of the morphism 
 $\phi_k$ and set $V^{k+1}_{i+1}:= \Ker H^{k+2}(\phi_{ki})$ with transgression given by the inclusion  
 $[d]:V^{k+1}_{i+1} \hookrightarrow H^{k+2}(\wedge (V^{\leq k} \oplus V^{k+1}_{\leq i}))$.
 Define $V^{k+1}= \oplus_{i \geq 0} V_i^{k+1}$.
 
 It remains to see that the extension just defined satisfies the property claimed in the lemma.
 Let $v= \sum_{i=0}^{n}v_i \in V^{k+1}, \; v_i \in V^{k+1}_i$. If $d(v)=d(\alpha)$, 
 for some $\alpha \in \mathcal{M}_k$, then 
 $d(v_n)=d(\alpha - \sum_{i=1}^{n-1}v_i) \in d( \wedge (V^{\leq k} \oplus V_{<n}^{ k+1}))$. 
 This forces $v_n=0$, if $n>0$. A downward induction on $n$ shows $v=v_0$.
\end{proof}

\begin{prop}
\label{k+2} 
A k-formal space  $M$ with $H^{\geq k+2}(M)=0$ is formal.
\end{prop}

\begin{proof} 

We use a result of  \cite{DGMS} to deduce the formality of $M$, corresponding to 
the case $k=\infty$ from our Proposition \ref{art}, Part \eqref{i}. That is, we have 
to find a minimal model $\mathcal{M}=(\wedge V,d)$ that admits a decomposition 
$V^i=C^i \oplus N^i$ for any $i$, such that $C^i= \Ker(d|_{V^i})$ and any closed element 
in the ideal generated by $\oplus N^i$, denoted  $I(\oplus_{i \geq 1} N^i)$, is exact in $\mathcal{M}$.

The $k$-formality is characterized on the other hand by Proposition \ref{art}, Part \eqref{i}, 
hence we choose $\mathcal{M}_k= (\wedge V^{\leq k}, d)$ a $k$-minimal model with a decomposition  
$V^i=C^i \oplus N^i$ for $i \leq k$ such that any closed element of degree at most $k+1$ in 
$N^{\leq k} \cdot  \mathcal{M}_k$ is exact in $\mathcal{M}_k$. 
Extend $\mathcal{M}_k$ to a minimal model $\mathcal{M}$, such that the $\mathcal{M}_{k+1}$ 
extension is constructed as in Lemma \ref{sem}.

For $i \geq k+2$ take $\alpha \in V^i, \; d(\alpha)=0$. Since $H^{\geq k+2}(\mathcal{M})=0$, 
there is $z \in \mathcal{M}$, such that $\alpha=d(z)$. 
Due to the decomposability of the differential 
of the minimal algebra $\mathcal{M}$, $\alpha=0$. Therefore 
$C^{\geq k+2}=0$ and we may take $N^{\geq k+2}=V^{\geq k+2}$.

For $i=k+1$, we know from Lemma \ref{sem} that $C^{k+1}= V_0^{k+1}$. 
Take  $N^{k+1} = \oplus_{i>0}V^{k+1}_i$.

It remains to see that every closed homogeneous element $\alpha \in I(\oplus_{i \geq 1} N^i)$, 
is exact in $\mathcal{M}$.
If $\alpha$ has degree $\geq k+2$, this is clear, since $H^{\geq k+2}(\mathcal{M})=0$. 
If $\alpha$ is of degree $\leq k$, then $\alpha \in N^{\leq k} \cdot \mathcal{M}_k$ and 
$\alpha$ must be exact in $\mathcal{M}_k$, hence in $\mathcal{M}$, 
by Proposition \ref{art}, Part \eqref{i}.

The last case to be solved is when $\alpha$ is homogeneous of degree $k+1$.  In this case, 
$\alpha=\alpha_1+\alpha_2$, with $\alpha_1 \in (N^{\leq k} \cdot \mathcal{M}_k)^{k+1}$ and 
$\alpha_2 \in N^{k+1}$.  Then $d(\alpha)=0$ implies $d(\alpha_2)=d(-\alpha_1)$, hence 
$\alpha_2=0$, by Lemma \ref{sem}. The exactness of $\alpha= \alpha_1$ follows again 
by Proposition \ref{art} \eqref{i}. 
\end{proof}

We remark that \cite[Lemma 2.10]{FM} is a consequence of the above result.

\begin{corollary}
\label{new}
Complex plane projective curve complements are formal spaces.
\end{corollary}

\begin{proof}
One knows that the complements of plane projective curves are 1-formal spaces, 
having the homotopy type of a CW-complex with cells of dimension $\leq 2$; see 
\cite{K} and \cite{DP} respectively, for more details. Hence their cohomology is trivial 
in dimension $\geq 3$ and we  can apply Proposition \ref{k+2}.
\noindent \end{proof}

The same result was proved in \cite{CM}, using a different approach.

\section{Partial formality and resonance}
\label{section:partial formality}
 
\begin{definition}
\label{var-res}
Let  $M$ be a (finite type) connected CW complex. Define the resonance variety  
$\mathcal{R}^{q}_{k}(M)$ as the subset (homogeneous subvariety) of all cohomology classes 
$w \in H^{1}(M, \k)$ such that  
$\dim_{\k} H^{q}(H^*(M), \mu_{w}) \geq k$, where $\mu_{w}$ is 
the differential induced by the multiplication by $w$. 
\end{definition}

To obtain another type of obstruction to formality, related to resonance varieties, 
we begin with a lemma on bigraded minimal models.

\begin{lemma}
\label{bigrad}
Let $H^*$ be a connected  graded-commutative algebra and denote by $\mathcal{B}$ its 
bigraded minimal model as defined and constructed in \cite{HS}.  If 
$\mathcal{R}_1^q(H^*) \nsubseteq \{0\}$, 
then the vector space $\mathcal{B}^q$ has infinite dimension.
\end{lemma}

\begin{proof}

We evaluate first the resonance varieties of $H^* \cong H^*(\mathcal{B})$.

 Take $[\omega] \in H^1(\mathcal{B})$  and $[\eta] \in H^q (\mathcal{B})$ with 
 $d(\omega)=0, \;d(\eta)=0$ and $[\omega \eta]=0$ in $H^{q+1}(\mathcal{B})$. Note that
 $\omega \in Z_0^1$: write $ \omega= \sum _{i \geq 0}\omega_i, \; \omega_i \in Z^1_i$, 
 with almost all $\omega_i=0$. Now $d(\omega)=0$ implies $d(\omega_i)=0,\; \forall i$; 
 moreover, for $i>0$, all $\omega_i$ are boundaries, as follows from \eqref{basic}, 
 so $\omega_i=0$, by minimality.
 The same type of argument  shows that we can take $\eta \in \mathcal{B}_0^q$.

The equality $[\omega \eta]=0$ in $H^{q+1}$  translates into $\omega \eta =d(\alpha)$ 
in $\mathcal{B}^{q+1}$, for some $\alpha \in \mathcal{B}^{q}_1$. Now suppose $\omega \neq 0$ 
and $[\eta] \notin [\omega]H^{q-1}(\mathcal{B})$.

If $\alpha =0$, then $\omega \eta=0$ in the free graded commutative algebra $\mathcal{B}$, 
hence $\eta = \overline{\eta} \omega$ with $\overline{\eta} $  of lower degree 0, 
hence $d(\overline{\eta})=0$. But this implies $[\eta]= \pm [\omega][\overline{\eta}]$, a contradiction.

So,  $\alpha:= \alpha_1 \neq 0$. We construct a sequence 
$(\alpha_i)_{i \geq0}, \;\alpha_i \in \mathcal{B}^q_i, \; \alpha_0:=\eta$, such that 
$d(\alpha_{i+1})=\omega \alpha_{i}$ and $\alpha_i \neq 0 \;\forall i$. Assume $\alpha_n$ 
already defined and satisfying the required properties for some fixed $n \geq 1$. Then 
$d( \omega \alpha_n)=0$, so there is $\alpha_{n+1} \in \mathcal{B}^q_{n+1}$ such that 
$d(\alpha_{n+1})= \omega \alpha_{n}$. Also $\alpha_{n+1}=0$ implies $\omega \alpha_{n}=0$. 

The last equality will lead to a contradiction.
 If $\omega \alpha_{n}=0$, then $\alpha_n= \omega \beta, \; \beta \in \mathcal{B}^{q-1}_n$. 
 Differentiating the last equality we obtain $\omega \alpha_{n-1} = \pm \omega d(\beta)$, 
 hence $\alpha_{n-1} \pm d(\beta)= \omega \zeta_{n-1}$, with  $\zeta_{n-1} \in \mathcal{B}_{n-1}^{q-1}$.
 Differentiating again we get $\omega \alpha_{n-2} = \pm \omega d(\zeta_{n-1})$, which implies 
 $\alpha_{n-2} \pm d(\zeta_{n-1})= \omega \zeta_{n-2}$ with $\zeta_{n-2} \in \mathcal{B}_{n-2}^{q-1}$. 
 A downward induction shows that eventually we obtain the equality 
 $\eta \pm d(\zeta_{1})= \omega \zeta_{0}$, with $\zeta_{1} \in \mathcal{B}_{1}^{q-1}$ and 
 $\zeta_{0} \in \mathcal{B}_{0}^{q-1}$, which implies $[\eta]=[\omega][\zeta_0]$, again a contradiction.
 
  This ends the construction of the sequence $(\alpha_i)_{i \geq0}, \;\alpha_i \in \mathcal{B}^q_i$ , 
  with nonzero elements $\alpha_i$ and the conclusion follows.
\end{proof}

{\bf Proof of Theorem \ref{zerores}}
Let us denote by $\mathcal{B}$ the bigraded minimal model of the cohomology algebra 
$H^*(G)$, and by $_{s}\mathcal{B}$ the $s$-minimal model of $H^*(G)$.
According to the previous lemma, it is enough to show that $\mathcal{B}^q$ 
has finite dimension, for $q \leq s$. The $s$-formality of $G$ implies that the $s$-minimal model 
of $G$ coincides with the  $s$-minimal model of $H^*(G)$, that is $\mathcal{M}_s(G) =\;_{s}\mathcal{B}$. 
Since $G$ is finitely generated and nilpotent, $\mathcal{M}(G)=\mathcal{M}_1(G)=\mathcal{M}_s(G)$ 
is a finite dimensional vector space. Clearly, $\mathcal{B}^q= _s\mathcal{B}^q$, for $q \leq s$. 
Our proof is complete. \hfill $\square$

\begin{example}
\label{arr}
The nilpotency condition in Theorem \ref{zerores} is necessary. We consider 
$G_{\mathcal{A}}=\pi_1(M_{\mathcal{A}})$ to be the fundamental group of 
the complement of a central complex hyperplane arrangement 
$\mathcal{A} \subset \mathbb{C}^n$, $n \geq 3$.
Then $G_{\mathcal{A}}$ is finitely generated (see for instance \cite{OT}) and 1-formal 
($M_{\mathcal{A}}$ is a formal space -see for example \cite{OT}- hence 1-formal, 
so $G_{\mathcal{A}}$ is also 1-formal). 

 In this setting, the properties below are equivalent:
 \begin{enumerate}
 \item \label{e1}
  The hyperplanes of $\mathcal{A}$ are in general position in codimension 2.
 \item \label{e2}
  The group $G_{\mathcal{A}}$ is abelian.
 \item \label{e3}
  The group $G_{\mathcal{A}}$ is nilpotent.
 \item \label{e4}
 $\dim_{\mathbb{Q}} \gr(G_{\mathcal{A}}) \otimes \mathbb{Q} < \infty$.
 \item \label{e5}
  $\mathcal{R}_1^1(G_{\mathcal{A}}) \subseteq \{0\}$.
 \item \label{e6}
  $\mathcal{V}_1^1(G_{\mathcal{A}}) \subseteq \{1\}$. 
 \end{enumerate}
 (Here $\gr(G) \otimes \mathbb{Q}$ is the rational associated graded Lie algebra of
 a group $G$, and $\mathcal{V}_1^1(G)$ denotes its (first) characteristic variety in
 degree one.)
 
The implication \eqref{e1} $\Rightarrow$ \eqref{e2} follows from Hattori's Theorem 
from \cite{Hat}, and 
\eqref{e2} $\Rightarrow$ \eqref{e3}  $\Rightarrow$ \eqref{e4} are obvious. 
For \eqref{e4} $\Rightarrow$ \eqref{e1}, we refer to \cite[Proposition 2.12]{F}. The implication 
\eqref{e3} $\Rightarrow$ \eqref{e5} is given by our Theorem \ref{zerores} and 
\eqref{e5} $\Rightarrow$ \eqref{e1} is implicit in the proof of Proposition 2.12, 
from \cite{F}. \eqref{e3} $\Rightarrow$ \eqref{e6} follows from \cite[Theorem 1.1]{MP}. 
Finally, \eqref{e6} $\Rightarrow$ \eqref{e5}, since in the 1-formal case 
(which is the case for $G_{\mathcal{A}}$) we have a local isomorphism  
$\mathcal{R}_1^1 \cong \mathcal{V}_1^1$ given by the exponential map; see \cite[Theorem A]{DPS}.
\end{example}

\begin{remark}
The partial formality in the hypothesis of Theorem \ref{zerores} is also
essential.  The more precise claim is that for any finite connected CW-complex $M$, 
there is a finitely generated 2-step nilpotent group $G$ such that 
$\mathcal{R}_k^1(G)=\mathcal{R}_k^1(M)$, for all $k$. See \cite{MP}.
Now take $M=M_{\A}$, where the arrangement $\A$ is not in general position in 
codimension 2. Since $\mathcal{R}_*^1(G_{\A})=\mathcal{R}_*^1(M_{\A})$, by
Definition \ref{var-res}, it follows from Example \ref{arr} that $\RR_1^1(G)\not\subseteq \{ 0\}$. 

Notice that the case of resonance varieties of finitely generated nilpotent groups 
(even restricted to 2-step nilpotent groups) is very different 
from the one of characteristic varieties, 
described in \cite{MP}.
\end{remark}

Given a graded-commutative algebra $H^*$, let $K$ be the kernel of 
the multiplication map $\mu: H^1 \wedge H^1 \longrightarrow H^2$, called in
\cite{CT} the characteristic subspace of $H^*$.
 
 \begin{lemma}
 \label{non}
 The subspace $K$ contains  no nontrivial decomposables if and only if 
 \begin{equation}
 \label{incl}
 \{ \omega \in K | \omega^2=0 \in \wedge^4 H^1\} \subseteq \{0\}
 \end{equation}
 Both properties are equivalent to $\mathcal{R}_1^1(H^*) \subseteq \{0\}$.
 \end{lemma}

 \begin{proof}
 Assume there is $0 \neq \omega = \sum_{i=1}^m x_i y_i$, written in canonical form, 
 satisfying $\omega^2=0$. Since 
 $\omega^i= i! \sum_{1 \leq k_1< \dots <k_i \leq m} x_{k_1}y_{k_1} \dots x_{k_i}y_{k_i} \neq 0$ 
 if $i \leq m$  we conclude that $m+1=2$, that is $\omega$ is a nontrivial decomposable. 
 
 The converse implication is immediate: a decomposable element $ \omega= \alpha_1 \alpha_2 \neq 0$ 
 satisfies $\omega^2=0$.
 
 It is equally easy to see that the fact that $K$ does not contain nontrivial decomposables 
 is equivalent to  $\mathcal{R}_1^1(H^*) \subseteq \{0\}$:

 Assume there is $\alpha \in \mathcal{R}_1^1(H^*),\; \alpha \neq 0$, hence there exists 
 $\beta \in H^1,\; \alpha \wedge \beta \neq 0$ such that $\alpha \beta =0 $ in $H^2$. 
 This implies $\alpha \wedge \beta$ is a nontrivial decomposable element of $K$.
 
 On the other hand, if for some $\alpha, \; \beta \in H^1$ we have 
 $0 \neq \alpha \wedge \beta \in K$, then $\alpha \beta =0$ in $H^2$, hence 
 $\mathcal{R}_1^1(H^*)$ contains nonzero elements.
 \end{proof}
 
 We see in the next Example that the triviality of the resonance varieties, up to a degree, 
 is not a sufficient condition for partial formality,
 even when the cohomology is partially generated in degree one.
 
\begin{example}
\label{contr}
Let $\mathcal{B}= \wedge(x_1, x_2, y_1, y_2, z, \omega_1, \omega_2, \alpha)$, 
$d(x_i)=d(y_i)=d(z)=0,\; d(\omega_1)=x_1y_1+x_2z,\;d(\omega_2)=x_2y_2+x_1z$,
$d(\alpha)=x_1\omega_1+x_2 \omega_2$, be the $1$-formal minimal model generated 
in degree $1$ from \cite{CT}, Example 2.8.

Then $\mathcal{B}$, in the notations from 
the beginning of Section \ref{obstr}, admits a bigrading  compatible with $d$,
with $Z_0=<x_i, y_i, z>,\; Z_1=<\omega_1, \omega_2>, \; Z_2=<\alpha>$.
One can check that $H_+^{\leq 2}(\mathcal{B})=0$.
Consider now a "deformation" of $\mathcal{B}, \;\mathcal{M}=\wedge(Z_0 \oplus Z_1 \oplus Z_2)$, 
with differential $D$ defined as follows: $D|_{Z_0,Z_1} = d|_{Z_0,Z_1}$ and 
$D(\alpha)=d(\alpha)+p, \; p \in \wedge^2Z_0$.  
A direct computation shows that  $D^2 = 0$.

Next we will prove that $H^{\leq 2}\mathcal{M} \cong H^{\leq 2}\mathcal{B}$, 
as algebras. As explained in \cite{S}, $\B$ is the minimal model of a finitely presentable
(3-step) nilpotent group $G_{\B}$.
Hence, via Proposition \ref{art}, Part \eqref{iii} and Theorem \ref{zerores}, 
applied for $M=K(G_{\B}, 1)$, one finds that $H^{\leq 2}\mathcal{M}$ is generated in degree $1$ 
and $\mathcal{R}_1^1(\mathcal{M}) \subseteq \{0\}$.

Obviously $H^1 \mathcal{M}= H^1 \mathcal{B}=Z_0$. Take $w=\overline{w}+ \xi \alpha$ such that 
$d(w)=0, \; \overline{w} \in \wedge^2 (Z_0 \oplus Z_1)$ and $\xi \in \wedge^1 (Z_0 \oplus Z_1)$. Then 
$d(w)=d(\overline{w}) + d(\xi) \alpha - \xi d(\alpha)=0$ implies $d(\xi)=0$, hence 
$\xi \in \wedge^1(x_i, y_i,z)$. It follows that 
\begin{equation}
\label{*}
d(\overline{w}) - \xi x_1 \omega_1 - \xi x_2 \omega_2=0.
\end{equation}
Notice that there are monomials in $d(\overline{w})$ containing $\omega_i$ if and only if 
$\overline{w}$ has a monomial $a \omega_1 \omega_2$with  $a \neq 0$; 
grouping the monomials in \eqref{*} which contain $\omega_1$ (respectively $\omega_2$), 
we conclude that $\xi x_1 =0$ and $\xi x_2=0$, which is possible only when $\xi=0$.

So $d(w)=0$ implies $w \in \wedge^2(x_i, y_i, z, \omega_i)$. The same way $D(w)=0$ implies 
$w \in \wedge^2(x_i, y_i, z, \omega_i)$, hence $d(w)=0 \Leftrightarrow D(w)=0$, 
so $w$ is a cocycle in $\mathcal{B}$ if and only if $w$ is a cocycle in $\mathcal{M}$. 
Since $\im d|_{\mathcal{B}^{1}} \cong \im D|_{\mathcal{M}^{1}}$, 
we obtain a vector space isomorphism:
\begin{equation}
\label{B=M}
H^{\leq 2}\mathcal{M} \cong H^{\leq 2}\mathcal{B}
\end{equation}

We still need a graded algebra isomorphism. Using the previous notations, consider 
the graded algebra $\mathcal{C}^*:=\frac{\wedge^*Z_0}{\wedge^*Z_0\cdot dZ_1}$. Notice that 
we have $H^{\le 2}\mathcal{B} \cong \mathcal{C}^{\le 2}$, as algebras, 
since $H_+^{\leq 2}(\mathcal{B})=0$. Define a graded algebra morphism
\begin{equation}
\label{GAlg}
\psi : \mathcal{C}^* \longrightarrow H^*(\mathcal{M})
\end{equation}
given by $\psi(z)=[z], \; \forall z \in Z_0$. It is clear that $\psi^1$ is
a linear isomorphism, and $\psi^2$ is a surjection between vector spaces of the same dimension, 
according to the above computations (see \eqref{B=M}).

However, we will see that the 1-minimal model $\mathcal{M}$ is not $1$-formal, if $p=y_1 y_2$. 
Assuming the contrary, we have an isomorphism of D.G. algebras 
$\phi: \mathcal{B} \longrightarrow \mathcal{M}$ 
(since $\mathcal{B}$ is the $1$-minimal model of the cohomology algebra of $\mathcal{M}$).

Moreover, we can choose $\phi$ such that $\phi|_{Z_0}=\id$. Indeed, 
let $\tilde{h}: \mathcal{B} \overset{\sim}\longrightarrow \mathcal{B}$ be 
the 1-minimal model of the graded algebra automorphism induced by 
$\phi, \; h:H^{\leq 2}(\B)\overset{\sim}\longrightarrow H^{\leq 2}(\mathcal{M}) \equiv H^{\leq 2}(\B)$. 
Replacing an arbitrary $\phi$ by $\phi \circ \tilde{h}^{-1}$, we obtain the desired property.
 
Checking the equality $D \phi =\phi d$ on the generators $\omega_i$, one gets 
$\phi(\omega_i) - \omega_i \in \wedge^1(x_i, y_i, z)$. Clearly, 
$\phi(d(\alpha))$ is a sum of monomials, each containing either $x_1$ or $x_2$.
Let $\phi(\alpha)=a_1x_1+a_2x_2+b_1y_1+b_2y_2+cz+d_1 \omega_1+d_2 \omega_2+e \alpha$, 
with $a_i, b_i, c, d_i, e \in \k$. Notice that $e \neq 0$, otherwise 
$\phi :\mathcal{B}^1 \rightarrow \mathcal{M}^1$ would not be a linear isomorphism. 
Then $D(\phi(\alpha))$ necessarily contains the monomial $y_1y_2$, with nontrivial coefficient $e$. 
But this contradicts the fact that $\phi$ is a D.G.A. morphism.
\end{example}

\begin{remark}
\label{rem=noteq}
Let $G$ be a finitely presentable group, with 1-minimal model $\mathcal{M}$. In Lemma 3.17 
on page 35, implication $(i) \Rightarrow (ii)$, the authors of \cite{ABC} note that
$H^2(\mathcal{M})= (H^1(\mathcal{M}))^2$, if $G$ is 1-formal. This can be recovered from our
Proposition \ref{art}\eqref{i}-\eqref{ii}, case $k=1$.

On the other hand, this implication cannot be reversed, contrary to the claim from 
\cite[Lemma 3.17]{ABC}. Indeed, the 1-minimal model $\mathcal{M}$ constructed in 
Example \ref{contr} above can be realized as $\mathcal{M}= \mathcal{M}(G) =\mathcal{M}_1(G)$,
where $G$ is a finitely presentable (3-step) nilpotent group, by the general theory from
\cite{S}. It follows from Example \ref{contr} that 
$H^2(\mathcal{M})= (H^1(\mathcal{M}))^2$, yet $G$ is not 1-formal. 
\end{remark}

\section{Heisenberg-type groups}
\label{section:heisenberg}

\begin{definition}
\label{heisenberg}
The integral Heisenberg group $\mathcal{H}_n$ is given by the central extension 
\begin{equation}
\label{sir}
0 \longrightarrow \Z \longrightarrow \mathcal{H}_n \longrightarrow \Z^{2n} \longrightarrow 0,
\end{equation}
corresponding to the cohomology class 
$\omega \in H^2(\mathbb{Z}^{2n}, \mathbb{Z})=\wedge^2_{\mathbb{Z}}(x_1, y_1,\dots, x_n, y_n)$, 
where $\omega=x_1 \wedge y_1+ \dots +x_n \wedge y_n$.
\end{definition}

It is immediate that the minimal model of $\mathcal{H}_n$ is the minimal DG algebra  
generated in degree 1,  $\mathcal{M}=\wedge(x_{1}, y_{1}, \dots ,x_{n}, y_{n}, z)$, 
with differential $d(x_{i})=d(y_{i})=0, \forall i$ and $d(z)=x_{1}\wedge y_{1} +\dots +x_{n}\wedge y_{n}$.
Note that the multiplication by 
$\omega$ in the exterior algebra $E^*=\wedge^*(x_{1}, y_{1}, \dots ,x_{n}, y_{n})$, 
$E^i \overset {\mu_{\omega}}  \longrightarrow E^{i+2}$, is injective for $i\leq n-1$, 
by the hard Lefschetz theorem, see \cite{We}.

\begin{lemma}
\label{1-formality}
The cohomology of the Heisenberg group $\mathcal{H}_{n}$ is given by 
\begin{equation}
\label{H^q}
H^{q}(\mathcal{H}_{n}) \cong \frac{\wedge^{q}(x_i, y_i)}{\omega \wedge^{q-2}(x_i, y_i)} 
\oplus \{\eta z \; |  \; \eta \omega =0, \; \eta \in \wedge^{q-1}(x_i,y_i) \} , \; \forall q.
\end{equation} 
The second summand is trivial, for $q \leq n$, and non-trivial, for $q=n+1$.
\end{lemma}

\begin{proof}
It is clear  that 
\begin{equation}
\label{H^1}
H^{1}(\mathcal{H}_{n})=\wedge^1(x_i, y_i) 
\end{equation}
Let us compute $H^q(\mathcal{H}_n)$, for $2 \leq q$.
Any $q$-form $\xi\in \wedge^{q}(x_i, y_i, z)$ is of the type $\xi = \eta_1 +\eta_2z$, where 
$\eta_1 \in \wedge^{q}(x_i,y_i)$ and $\eta_2 \in \wedge^{q-1}(x_i,y_i)$. Hence 
$\xi$ is a cocycle if and only if $ \pm d(\xi)=\eta_2 \omega =0$. In case $q \leq n$ 
the last equality implies $\eta_2=0$. Moreover we get that any $q$-coboundary is of the type 
$\eta_2 \omega$,  $\eta_2 \in \wedge^{q-2}(x_i,y_i)$. 
Consequently, $H^q(\mathcal{H}_n)$ has the asserted form. Clearly, $\eta=y_1 \cdots y_n$ 
creates a nontrivial contribution of the second summand, in degree $n+1$.
\end{proof}

\begin{remark}
\label{FernandezMunoz}

The above Lemma shows that $\mathcal{H}_n$ is $(n-1)$-formal 
(use Theorem \ref{B}, Part \eqref{converse}), 
but not $n$-formal (as follows from Theorem \ref{B}, Part \eqref{motz}).
At the same time, it shows that the notion of partial formality from \cite{FM} is strictly stronger 
than the one used here.

Consider the 1-formal (in the sense of our Definition \ref{s-formality}) group $\mathcal{H}_2$. 
The 1-formality of $\mathcal{H}_2$ in the sense of  \cite[Definition 2.2]{FM} would imply
a decomposition of the space of degree 1 generators 
$<x_1, x_2, y_1, y_2, z>$ as a direct sum $C^1 \oplus N^1$, satisfying the conditions:\\
\textit{(i)} $d(C^1)=0$;\\
\textit{(ii)} the restriction of the differential $d$ to $N^1$ is injective;\\
\textit{(iii)} any closed element in the ideal generated by $N^1$ in 
$\mathcal{M}= \mathcal{M}_1(\mathcal{H}_2)= \mathcal{M} (\mathcal{H}_2)$ 
is exact in $\mathcal{M}(\mathcal{H}_2)$.

In this case $C^1$ is the subspace $<x_1, x_2, y_1, y_2>$ and $N^1$ 
must be generated by $z+\alpha$ with $\alpha \in \wedge^1(C^1)$. Plainly, 
$(z+ \alpha)x_1x_2y_1y_2$ is closed, but not exact.
\end{remark}
 
  Let us compute some resonance varieties for Heisenberg groups:

\begin{prop}
\label{H-char}
$\mathcal{R}^*_{1}(\mathcal{H}_n)=\{0\}$, for $* \leq n-1$ and 
$\mathcal{R}^n_{1}(\mathcal{H}_n)=\k^{2n}$. 
\end{prop}

\begin{proof}
We know the additive structure of $H^*(\mathcal{H}_n)$ (see Lemma \ref{1-formality}). 
We have to compute the cohomology of the complex $H^*(\mathcal{H}_n)$ with differential 
given by the multiplication with the class of an element $0 \neq \xi \in \wedge^1(x_i,y_i)$.

We may assume $\xi=x_1$, by a linear change of coordinates. The $n$-class $[y_1\dots y_n]$ 
cannot be obtained by multiplying some $(n-1)$-class by $[x_1]$: 
the equality $y_1\dots y_n= x_1 \eta + \omega \beta$, with $\eta  \in \wedge ^{n-1}(x_i, y_i)$ 
and $\beta \in \wedge ^{n-2}(x_i, y_i)$ is impossible, since at the right hand side 
all monomials contain some $x_i$ component. 
 
 However $ x_1 y_1 \dots y_n=\omega  y_2 \dots y_n$, hence $[x_1y_1 \dots y_n]=0$ 
 in $H^{n+1}(\mathcal{H}_n)$; so $[x_1] \in \mathcal{R}^n_{1}(\mathcal{H}_n)$. This proves 
 that $\mathcal{R}^n_{1}(\mathcal{H}_n)=\wedge^1 (x_i, y_i)$.

We can apply Theorem \ref{zerores} 
and deduce $\mathcal{R}^{*}_{1}(\mathcal{H}_n) \subseteq \{0\}$, for $*<n$.

It remains to see that $0 \in \mathcal{R}^{*}_{1}(\mathcal{H}_n) $, if $*<n$. 
Assuming the contrary, it follows from Lemma \ref{1-formality} that $H^n(\mathcal{H}_n)=0$. 
But this implies $\mathcal{R}^{n}_{1}(\mathcal{H}_n)= \emptyset$, 
contradicting the first computation.
\end{proof}

However, the triviality of the resonance varieties and the property of having 
the cohomology generated in degree 1 are independent in general, as we can see below:

\begin{example}
\label{initial}
Let $G$ be a finitely generated 2-step nilpotent group with minimal model generated in degree 1, 
$\mathcal{N}=\wedge(x_1,x_2, y_1, y_2, z)\otimes \wedge (\omega_1, \omega_2)$ with differentials 
$dz=dx_i=dy_i=0,\; i=1,2$ and $d\omega_1=x_1y_1+x_2z, \;d\omega_2=x_2y_2+x_1z$; 
see Example \ref{contr}. Then 
$\mathcal{R}^1_1(G)=\{0\}$, but $H^2(G) \neq (H^1(G))^2$.

Set $K:=K(G,1)$.  
First we compute the cohomology of $K$ in low degrees. 
It is immediate that $H^1(K)=\wedge^1(x_i, y_i, z)$ and 
$H^2(K)= H_0^2(K)\oplus H_1^2(K)\oplus H_2^2(K)$, where 
$H_0^2(K) = \frac{\wedge^2(x_i, y_i, z)}{<d\omega_1, d\omega_2>}$. A direct computation
shows that $H_1^2(K) = <x_1 \omega_1+x_2\omega_2>$, and $H_2^2(K) =0$.

Hence 
$H^2(K)=\frac{\wedge^2(x_i, y_i, z)}{<d\omega_1, d\omega_2>}\oplus <x_1 \omega_1+x_2\omega_2> \neq (H^1(K))^2$.

As for the resonance variety $\mathcal{R}^1_1(K)$, take a one-cycle $\xi \neq 0$  and  
a one-cycle $\eta$ such that $[\eta\xi]=0$  in cohomology. This can happen only if 
$\eta\xi=a d\omega_1+b d\omega_2$ for some $a,b \in \k$.

If $\eta\xi=a d\omega_1+b d\omega_2$, then 
$0=(\eta \xi)^2= a^2(d\omega_1)^2+ 2ab d \omega_1 d \omega_2+b^2(d\omega_2)^2$, i.e. 
$0=2a^2x_1y_1x_2z+ 2ab x_1y_1x_2y_2+2b^2x_2y_2x_1z$, so $a=b=0$. Consequently  
$\eta \xi=0$ in $\wedge(x_1,x_2, y_1, y_2, z)$. Since $\xi \neq 0,\;\eta \in <\xi>$.

Therefore, $\mathcal{R}^1_1(K)=\{0\}$.
\end{example}

 Let $G$ be the 2-step nilpotent group defined by the central extension:
\begin{equation}
\label{sir'}
0 \longrightarrow B\longrightarrow G \longrightarrow A \longrightarrow 0,
\end{equation}
where $B$ is an abelian group of rank 1, and 
$A$ is an abelian group of finite rank $n$.

The minimal model of $G$ from \eqref{sir'} is of the form 
$\mathcal{M}(G)=\wedge(t_1, \dots, t_n) \otimes \wedge(z)$, with differentials 
$d(t_i)=0, \; \forall i$ and $d(z):= \omega \in \wedge^2(t_1, \dots, t_n)$.

\begin{definition}
\label{nilp}
$G$ is called of Heisenberg-type if $\omega \neq 0$.
\end{definition}

Moreover we may assume $\omega$ has the canonical form  $\omega= x_1y_1+ \dots +x_my_m$, 
where $2m=\rk(\omega)$; consequently:
\begin{equation}
\label{htype}
\mathcal{M}(G)=\mathcal{M}({\mathcal{H}_{m}}) \otimes (\wedge(t_{2m+1}, \dots, t_n),d=0).
\end{equation}

To obtain information on the resonance varieties associated to Heisenberg-type groups, 
we will use the following result:

\begin{prop}
\label{ref}
Let $A^*,\; B^*$ be graded-commutative connected algebras. Then
\begin{equation}
\label{R^q}
\mathcal{R}_1^q(A^*\otimes B^*)=\bigcup _{m+n=q}\mathcal{R}^m_1(A^*) \times \mathcal{R}^n_1(B^*)
\end{equation}  
\end{prop}

\begin{proof}
Set $C^*=A^* \otimes B^*$. If $\xi=\xi_A+\xi_B \in C^1=A^1 \oplus B^1$ is 
an arbitrary degree 1 element, then the multiplication by $\xi$ on $C^*$ is given by:
\begin{equation}
\label{kun}
\mu_{\xi}(a\otimes b)=\mu_{\xi_{A}}(a)\otimes b+(-1)^{|a|}a\otimes \mu_{\xi_B}( b), 
\; a \in A^*, \;b \in B^* \, .
\end{equation}
Consequently, by K\"{u}nneth, 
\begin{equation}
\label{H}
H^q(C^*, \mu_\xi)= \oplus_{m+n=q} H^m(A^*, \mu_{\xi_A}) \otimes H^n(B^*, \mu_{\xi_B})\, .
\end{equation}

By definition, $\mathcal{R}^q_1(A^* \otimes B^*)= \{ \xi \in A^1 \oplus B^1 \; |\;
H^q(A^* \otimes B^*, \mu_{\xi}) \neq 0\}= \\
\{ (\xi_A, \xi_B) \in A^1 \times B^1| \oplus_{m+n=q} H^m(A^*, \mu_{\xi_A}) \otimes H^n(B^*, \mu_{\xi_B}) \neq 0 \}=\\
\{ (\xi_A, \xi_B) \in A^1 \times B^1| \;\exists (m,n)\, , m+n=q, \;H^m(A^*, \mu_{\xi_A})\neq 0 \; \& 
\; H^n(B^*, \mu_{\xi_B}) \neq 0\}= \\
\{(\xi_A, \xi_B) \in A^1 \times B^1|\; \exists (m,n)\, , m+n=q, \;\xi_A \in \mathcal{R}^m_1(A^*) \; 
\& \; \xi_B \in \mathcal{R}^n_1(B^*)  \}= \\
\bigcup _{m+n=q}\mathcal{R}^m_1(A^*) \times \mathcal{R}^n_1(B^*)$.
\end{proof}

\begin{corollary}
\label{gr}
Let $G$ be a Heisenberg-type group, with minimal model described in \eqref{htype}. 
Then $\mathcal{R}^*_{1}(G)=\{0\}$, for $* \leq m-1$ and $\mathcal{R}^m_{1}(G)=\k^{2m}$.
\end{corollary}

\begin{proof}
By \eqref{htype} and Proposition \ref{ref}, $\mathcal{R}^k_1(G)$ equals
\begin{center}
$\mathcal{R}^k_1(H^*(\mathcal{H}_{m}) \otimes \wedge^*(t_{2m+1}, \dots, t_n))= 
\bigcup_{p+q=k}\mathcal{R}^p_1(\mathcal{H}_m) \times \mathcal{R}^q_1(\wedge^*(t_{2m+1}, \dots, t_n))$.
\end{center}
Since 
the resonance for the exterior algebra is trivial, the corollary follows from Proposition \ref{H-char}.
\end{proof}

\begin{corollary}
\label{c1}
A Heisenberg-type group $G$  with $\rk (\omega)=2m$ is $(m-1)$-formal, but not $m$-formal.
\end{corollary}
\begin{proof}

One can easily see from Theorem \ref{zerores}, via Corollary \ref{gr}, that $G$ is not $m$-formal. 
The $(m-1)$-formality follows from Theorem \ref{B}, Part \eqref{converse} and Lemma \ref{1-formality}, 
using again \eqref{htype} to describe the cohomology ring of $G$ up to degree $m$.
\end{proof}

\begin{corollary}
\label{T}
A Heisenberg-type group $G$ with $\rk (\omega)=2m$ cannot be realized as the fundamental group 
of a smooth projective complex variety $M$ with $\pi_{ \leq m}(\tilde{M})=0$, 
where $ \tilde{M}$ is the universal covering of $M$.
\end{corollary}

\begin{proof}
Assume $G= \pi_1(M)$, with $M$ a smooth projective complex variety. By the main result 
of \cite{DGMS}, $M$ is a formal space, hence $m$-formal, while $G$ is not $m$-formal 
(see Corollary \ref{c1}). Then there is $2 \leq i \leq m$ such that  
$\pi_i(\tilde{M}) \cong \pi_i(M) \neq 0$; see Theorem \ref{Thm}.
\end{proof}

\begin{ack}
I would like to thank Professor \c{S}. Papadima for his guidance during the elaboration of this work.
\end{ack}

\bibliographystyle{amsplain}

\end{document}